\documentclass[a4paper,reqno]{amsart}
\usepackage{amssymb}
\usepackage{latexsym}
\usepackage{amsmath}
\usepackage{euscript}
\usepackage{graphics,color}
\usepackage[all]{xy}
\usepackage[margin=3cm]{geometry}
\usepackage{MnSymbol} % for \udots

      \def\dC{{\mathbb C}}

   \def\dN{{\mathbb N}}   
      \def\dR{{\mathbb R}}

\newcommand{\be}{\begin{equation}}
\newcommand{\ee}{\end{equation}}
\newcommand{\ba}{\begin{eqnarray}}
\newcommand{\ea}{\end{eqnarray}}
\newcommand{\baa}{\begin{eqnarray*}}
\newcommand{\eaa}{\end{eqnarray*}}
\newcommand{\bb}{\begin{thebibliography}}
\newcommand{\eb}{\end{thebibliography}}

\newcounter{my}
\newcommand{\he}%
   {\stepcounter{equation}\setcounter{my}%
   {\value{equation}}\setcounter{equation}0%
   }%
\newcommand{\she}%
   {\setcounter{equation}{\value{my}}%
    }%

\newtheorem{theorem}{Theorem}[section]
\newtheorem{proposition}[theorem]{Proposition}

\newtheorem{lemma}[theorem]{Lemma}
\newtheorem{definition}[theorem]{Definition}
\theoremstyle{definition}

\newtheorem{remark}[theorem]{Remark}

\numberwithin{equation}{section}

\newcommand{\bra}[1]{\langle\,{#1}}
\newcommand{\ket}[1]{\mid{#1}\,\rangle}
\newcommand{\hg}[2]{\,\mbox{}_{#1}F_{ #2}\!}

\newcommand{\argu}[3]{\left(\begin{array}{c} #1\\#2\end{array} ; #3\right)}

\title{Continuous Representations of the Meta Hahn Algebra and Special Functions}
\author{Quentin Labriet and Luc Vinet}
\address{Centre de recherches math\'ematiques, Universit\'e de Montr\'eal, P.O. Box 6128, Centre-ville Station, Montr\'eal (Qu\'ebec), H3C 3J7}

\address{IVADO and Centre de recherches math\'ematiques, Universit\'e de Montr\'eal, P.O. Box 6128, Centre-ville Station, Montr\'eal (Qu\'ebec), H3C 3J7}

\begin{document}

\begin{abstract}
    Infinite families of biorthogonal rational functions are constructed using infinite dimensional representations of the three-generated meta Hahn algebra. Continuous Hahn and continuous dual Hahn polynomials are also recovered along the way using an embedding of the Hahn algebra inside the meta Hahn. Rational functions and polynomials arise as transition coefficients between eigenvectors and generalized eigenvectors, and their algebraic properties follow naturally from the relations of the meta Hahn algebra. Finally, biorthogonality relations are established for the rational functions introduced in the paper. 
\end{abstract}

\maketitle

\section{Introduction}
The present article builds on the program initiated in \cite{TVZ2025MetaHahn} which aims at studying biorthogonal rational functions (BRFs) via a representation theory point of view. Studying orthogonal polynomials (OPs) through representations of the Askey--Wilson algebra, and its specializations, has been extensively explored \cite{GranovskiiZhedanov86,granovskii1992mutual,Koornwinder2023charting}, and the program under development for BRFs follows the same approach using representation theory of the so-called \emph{meta algebras}. While several finite families of BRFs, including some $q$-analogues, have already been explored via finite dimensional representation theory \cite{bussiere2022bispectrality,BBPTTVZZ2026MetaqHahn,CLMTVZ2026MetaRacah,CTVZ2026MetaqRacah}, the present paper focuses on infinite dimensional representations. 

More precisely, we investigate OPs and BRFs that satisfy a bispectral property, i.e., they are solutions of both a recurrence equation and a differential or difference equation. In the case of OPs, two operators generate the so-called bispectral algebra, and in turn bispectral OPs appear as transition coefficients between eigenbases in a representation of this algebra. This is the approach followed in Section \ref{sec:continuousHahn} and \ref{sec:ContinuousDualHahn}. On the other hand, BRFs enjoy a similar algebraic structure but ordinary eigenvalue problems are replaced by \emph{generalized eigenvalue problems} (GEVP), i.e., problems of the form 
\begin{equation}
    X f =\lambda Z f,
\end{equation}
for a pair of operators $(X,Z)$. It was shown by Zhedanov \cite{zhedanov1999biorthogonal} that families of rational functions that are biorthogonal arise as solutions of generalized eigenvalue problem, and the bispectrality of the biorthogonal rational functions of Hahn type was given an algebraic description in \cite{vinet2021unified}. 

Finite families of OPs in the Askey--Wilson Scheme enjoy the bispectral property, with both operators being tridiagonal, and their bispectral algebra is the Askey--Wilson algebra, or its specializations. Following \cite{TVZ2025MetaHahn}, we are interested in BRFs where both operators $X$ and $Z$ are tridiagonal. The systematic study of such BRFs via representation theory was started in \cite{TVZ2025MetaHahn} with an example related to finite dimensional representations of the meta Hahn algebra, first described in \cite{vinet2021unified}. Subsequently, other meta algebras were explored such as the meta Racah \cite{CLMTVZ2026MetaRacah}, meta $q$-Hahn \cite{BBPTTVZZ2026MetaqHahn,bussiere2022bispectrality} and meta $q$-Racah \cite{CTVZ2026MetaqRacah} algebras. Alongside the meta algebras, the Leonard trios approach was developed \cite{CGLMVW2026trios} in the spirit of the Leonard pair theory for bispectral OPs \cite{terwilliger2003introduction}. The connection between the meta and the trio point of view was investigated and made explicit in the Hahn case \cite{CLMV2026trioHahn}. They turn out to be equivalent in this case, and are two faces of the same algebraic object. One advantage of the meta algebras' point of view is that they admit a presentation with only three generators with simple relations. This, in turn,  leads to simple representations where the three generators act bidiagonally. These features have been key to the study of BFRs in the previous works aforementioned. 

In this article, we first classify infinite-dimensional bidiagonal representations of the meta Hahn algebra. This classification yields two families of representations, referred to as being of \emph{positive} or \emph{negative type}, which are almost dual to each other. For each family, we will recover continuous (dual) Hahn polynomials as transition coefficients between two eigenbases, and a family of BRFs expressed in terms of a ${}_3F_2$ hypergeometric function. In both cases, we derive the recurrence relations and difference equations satisfied by the functions from the defining relations of the Hahn and meta Hahn algebras. Since our approach is algebraic in nature, the biorthogonality of the BRFs , which is an analytic feature in this case, are established via analytic methods. 
\smallskip

\paragraph{\bf{Outline of the paper}} The article follows the program developed in \cite{TVZ2025MetaHahn} and is organized along similar lines. Section \ref{sec:metaHahn} introduces the meta Hahn algebra and recall how the Hahn algebra is embedded into it. In Section \ref{sec:representations}, we turn to representation theory and describe two families of infinite dimensional bidiagonal representations. These representations will encode OPs and BRFs together with their algebraic properties. Section \ref{sec:Eigenvectors} is technical in nature and contains explicit expressions for (generalized) eigenvectors for the generators of the meta Hahn algebra together with the necessary action of some other operators. Section \ref{sec:continuousHahn} and \ref{sec:ContinuousDualHahn} describe how to recover continuous Hahn and continuous dual Hahn polynomials from the Hahn algebra and the representations introduced in Section \ref{sec:representations}. Finally, Section \ref{sec:BRFpositive} and \ref{sec:BRFnegative} present two families of BRFs together with their recurrence relations, difference equations and biorthogonality relations. 
\smallskip

\paragraph{\bf{Notations}} Let us fix some conventions. Throughout, $T^{\top}$ denotes the transpose of the operator $T$, acting on linear forms, no Hilbert space adjoint is ever considered. The Pochhammer symbol is $(a)_k=a(a+1)\cdots(a+k-1)$, $(a)_0=1$, and the generalized hypergeometric series are denoted as usual by
\begin{equation}
    \hg{p}{q}\argu{a_1,\dots,a_p}{b_1,\dots,b_q}{z}=\sum_{k\geq 0}\frac{(a_1)_k\cdots (a_p)_k}{(b_1)_k\cdots(b_q)_k}\frac{z^k}{k!}.
\end{equation}

%%%%%%%%%%%%%%%%%%%%%%
%%%%%%%%%%%%%%%%%%%%%%
%%%%%%%%%%%%%%%%%%%%%%
%%%%%%%%%%%%%%%%%%%%%%

\section{The Meta Hahn algebra}\label{sec:metaHahn}
We recall the definition and some properties of the meta Hahn algebra already introduced in \cite{vinet2021unified}. The results exposed in this section are extracted from \cite{TVZ2025MetaHahn}. 

\begin{definition} 
The meta Hahn algebra, denoted $m\mathfrak{H}$, is the algebra generated by three generators $X,\ Z$, and $V$ subject to the relations:
\begin{align}   
    [Z,X]&=Z^2+Z,\label{eq:relationsMetaHahn1}\\
    [X,V]&= \{V,Z\}+V+\xi I,\label{eq:relationsMetaHahn2}\\
    [V,Z]&=2X+\eta I\label{eq:relationsMetaHahn3},
\end{align}
where $[A,B]=AB-BA$ denotes the commutator, $\{A,B\}=AB+BA$ the anticommutator and $\xi,\eta$ are central parameters.   
\end{definition}
The Casimir element of $m\mathfrak{H}$ is given by  
\begin{equation}\label{eq:CasimirMeta}
    Q=\{V,Z^2+Z\}+2(X^2+Z^2)+2\eta X+2(\xi+1)Z.
\end{equation}

The Hahn algebra $\mathfrak{H}$ is generated by two elements $K_1$ and $K_2$ with defining relations
\begin{align}
    [K_1,[K_2,K_1]]&=aK_1^2+bK_1+c_1K_2+d_1I,\\
    [K_2,[K_1,K_2]]&=a\{K_1,K_2\}+bK_2+c_2K_1+d_2I,
\end{align}
where $a,\ b,\ c_1,\ c_2,\ d_1,\ d_2$ are central parameters with $a\neq 0$. A representation of this algebra is given by the bispectral operators of the Hahn polynomials. One central feature of the meta Hahn algebra, which gave its name, is that the Hahn algebra embeds in it. More precisely, the embedding is defined by:
\begin{equation}
    K_1=X+\mu Z, \ \ \ \ K_2=V, \ \ \text{with } \mu\in \dR.
\end{equation}
This can be shown using the defining relations of $m\mathfrak{H}$ \eqref{eq:relationsMetaHahn1}, \eqref{eq:relationsMetaHahn2}, \eqref{eq:relationsMetaHahn3} and the Casimir element \eqref{eq:CasimirMeta} (see \cite{TVZ2025MetaHahn}). A more exhaustive presentation of the algebraic properties of $m\mathfrak{H}$ can be found in \cite{CLMV2026trioHahn} where the link with the Leonard trios approach is made explicit.

%%%%%%%%%%%%%%%%%%%%%%
%%%%%%%%%%%%%%%%%%%%%%
%%%%%%%%%%%%%%%%%%%%%%
%%%%%%%%%%%%%%%%%%%%%%

\section{Infinite dimensional two-diagonal representations}\label{sec:representations}

We now turn to infinite dimensional representations of the meta Hahn algebras $m\mathfrak{H}$. We will present two families of representations where the three generators act in a two-diagonal way, i.e., there is a basis $\ket{n}$ such that the actions of $X,\ Z$ and $V$ only involve $\ket{n},\ \ket{n+1}$ or $\ket{n},\ \ket{n-1}$. We will see in Section \ref{sec:Eigenvectors} that this particular form makes it easy to solve GEVP and EVP. 

Our approach in this article is mostly algebraic, thus the representation space under consideration will be $\dC^\dN$ endowed with the product topology and its continuous dual space will then be the space of finitely supported sequences. The natural basis of this space will be denoted by $\ket{n}$ for $n\in \dN$, and we denote by $\bra{n^\dagger}|$ its dual basis. 

Using the defining relations of $m\mathfrak{H}$ \eqref{eq:relationsMetaHahn1}, \eqref{eq:relationsMetaHahn2}, and \eqref{eq:relationsMetaHahn3}, one can see that the only possibilities for two diagonal representations on $\dC^\dN$ are of the following two forms:
\begin{equation}
    \begin{array}{ccc}
        Z\ket{n}=c_n\ket{n}+a_n\ket{n-1}\\
        X\ket{n}=d_n\ket{n}+b_n\ket{n-1}\\
        V\ket{n}=e_n\ket{n}+f_n\ket{n+1}
    \end{array}
    \text{, or }\ \ \begin{array}{cc}
        Z\ket{n}=c_n\ket{n}+a_n\ket{n+1}\\
        X\ket{n}=d_n\ket{n}+b_n\ket{n+1}\\
        V\ket{n}=e_n\ket{n}+f_n\ket{n-1}
    \end{array}.
\end{equation}
Inspired by the terminology used for discrete series for $\mathfrak{su}(1,1)$, representations of the first case will be called of the \emph{positive type}, and the one of the second case of the \emph{negative type}. Representations of the positive type are explicitly given in the following proposition.
\begin{proposition}\label{prop:PostiveType}
    Two-diagonal irreducible representations of the positive type of $m\mathfrak{H}$ acting on $\dC^\dN$ are described by the following operators
\begin{align}
    Z\ket{n}&=\delta\ket{n}+a_n\ket{n-1},\\
    X\ket{n}&=(n(1+2\delta)+\alpha)\ket{n}+a_n(n+\beta)\ket{n-1},\\
    V\ket{n}&=-(n+\beta)(\beta+n+1)\ket{n}-\frac{(n+1)(2\alpha+\eta+n(1+2\delta))}{a_{n+1}}\ket{n+1},
\end{align}
with $a_n\neq 0$ for all $n\in \dN^*$, $a_0=0$, $\delta=0$ or $-1$, and where $\alpha,\beta \in \dC$ are related to the parameters of the meta Hahn algebra by
\begin{equation}\label{eq:Conditiion for irrep 2} 
\xi=-\beta(2\alpha+\eta-(\beta+1)(1+2\delta)). 
\end{equation}
and such that $2\alpha+\eta +n(1+2\delta)\neq 0$ for all $ n\in \dN$. 
\end{proposition}
\begin{proof}
    The proof will follow closely what was done for the finite dimensional situation in \cite{TVZ2025MetaHahn}. It is straightforward to see that this formulas define an irreducible representation of $m\mathfrak{H}$. Let us prove that they all have the given expressions. Start from 
    \begin{align}
        Z\ket{n}&=c_n\ket{n}+a_n\ket{n-1},\\
        X\ket{n}&=d_n\ket{n}+b_n\ket{n-1}.
    \end{align}
    Together with relation \eqref{eq:relationsMetaHahn1} one arrives at 
    \begin{equation}
        \frac{b_{n-1}}{a_{n-1}}-\frac{b_n}{a_n}=1,
    \end{equation}
    and then obtains
    \begin{equation}
        b_n=a_n (n+\beta). 
    \end{equation}
    Irreducibility imposes $a_n\neq 0$ for all $n\geq 0$. One also have $c_n(c_n+1)=0$ leading to $c_n=0$ or $-1$. Introduce
    \begin{equation}
         V\ket{n}=e_n\ket{n}+f_n\ket{n+1}.
    \end{equation}
    Using relation \eqref{eq:relationsMetaHahn3}, one finds that $f_n(c_n-c_{n+1})=0$ and since irreducibility implies $f_n\neq 0$ for all $n\geq 0$ we set $c_n=\delta$ with $\delta=0$ or $-1$. Then relation \eqref{eq:relationsMetaHahn1} leads to 
    \begin{equation}
        d_n=n(1+2\delta)+\alpha. 
    \end{equation}
    Next, using relation \eqref{eq:relationsMetaHahn3} we have 
    \begin{equation}
        f_n=f_{n-1}-2d_n-\eta,
    \end{equation}
    leading to 
    \begin{equation}
        f_n=-\frac{(n+1)(2\alpha+\eta+n(1+2\delta))}{a_{n+1}}.
    \end{equation}
    Finally, the formula for $e_n$ is obtained using relation \eqref{eq:relationsMetaHahn3} which leads to
    \begin{equation}
        e_n=-n(n-1+2\beta)+e_0. 
    \end{equation}
    Relation \eqref{eq:relationsMetaHahn2} will then enforce the condition \eqref{eq:Conditiion for irrep 2} and also fixes $e_0=-\beta(\beta+1)$. 
\end{proof}

Similarly, representations of the negative type are explicitly given in the following proposition whose proof is the same as for the positive type.

\begin{proposition}\label{prop:negativeType}
Two-diagonal irreducible representations of the negative type of $m\mathfrak{H}$ acting on $\dC^\dN$ are described by the following operators
\begin{align}
    Z\ket{n}&=\delta\ket{n}+a_n\ket{n+1}\\
    X\ket{n}&=(-n(1+2\delta)-\alpha)\ket{n}+a_n(-n+\beta)\ket{n+1}\\
    V\ket{n}&=(n-\beta)(\beta-n+1)\ket{n}-\frac{n(2\alpha-\eta+(n-1)(1+2\delta))}{a_{n-1}}\ket{n-1},
\end{align}
with $a_n\neq 0$ for all $n\in \dN$, $\delta=0$ or $-1$, $\alpha,\beta \in \dC$  are related to the parameters of the meta Hahn algebra by
\begin{equation}\label{eq:Conditiion for irrep 1} 
\xi=(\beta+1)(2\alpha-\eta+\beta(1+2\delta)). 
\end{equation}
and such that $2\alpha-\eta +(n-1)(1+2\delta)\neq 0 $ for all $ n\in \dN$. 
\end{proposition}

\begin{remark}\label{rem:Duality}
    The meta Hahn algebra is not stable under taking transpose, however the operators $Z^\top,V^\top,-X^\top$ satisfy relations \eqref{eq:relationsMetaHahn1}-\eqref{eq:relationsMetaHahn3} with parameters $\xi,-\eta$. Taking this fact into account, representations of the negative type are readily seen to be obtained from the positive type by transposition up to the change of  parameters $a_n\to a_{n+1}, \beta\to \beta+1, \alpha\to \alpha$.
\end{remark}
%%%%%%%%%%%%%%%%%%%%%%
%%%%%%%%%%%%%%%%%%%%%%
%%%%%%%%%%%%%%%%%%%%%%
%%%%%%%%%%%%%%%%%%%%%%

\section{(Generalized) Eigenvectors}\label{sec:Eigenvectors}
We will now describe here the solutions in the representation space of different GEVPs and EVPs. OPs and BRFs will then be construct as transition coefficients between different families of (generalized) eigenvectors. More precisely we will be interested in the eigenvectors of $V$, $W=X+\mu Z$ and the GEVP:
\begin{equation}\label{eq:Def GEVP}
  Xv=\lambda Zv.
\end{equation}

We will denote the eigenvectors of $V$ by $\ket{v_x^\pm}$ with the exponent $\pm$ referring to the positive or negative type of the representation under consideration and the index $x$ will be either continuous or discrete. Similarly, the eigenvectors of $W$ will be denoted $\ket{w_x^\pm}$ and the generalized eigenvectors of the GEVP \eqref{eq:Def GEVP} will be denoted $\ket{d_x^\pm}$. 

For future use, we also provide some needed actions of the operators $X,\ Z$ and $V$ on the different eigenbases. This will be necessary to study the properties of OPs and BFRs constructed in the next sections.  The statements of this section are all obtained through a direct computation.

\subsection{Representations of the positive type}

First, we consider a representation of the positive type as described in Proposition \ref{prop:PostiveType} and we find the following eigenvectors.

For $V$, the eigenvalues are of the form $\rho_n=-(n+\beta)(n+\beta+1)$ for all $n\in \dN$ and the eigenvectors are given by
\begin{equation}
\ket{v^+_n}=v_0^+(n)\sum_{k=n}^\infty (-1)^{k-n}\binom{k}{n}\frac{(2\alpha+\eta+(1+\delta)n+\delta (k-1))_{k-n}}{a_{n+1}\cdots a_k(2\beta+2+2n)_{k-n}}\ket{k}.
\end{equation}
The family $\{\ket{v^+_n}\}_{n\in \dN}$ forms a basis of $\dC^\dN$ and its dual basis $\bra {v_n^+}^\dagger|$ is given by
\begin{equation}\label{eq:Expression en*}
    \bra {v_n^+}^\dagger|=\frac{(2\alpha+\eta+\delta(n-1))_n}{v_0^+(n)a_1\cdots a_n (n+2\beta+1)_n}\sum_{k=0}^n a_1\cdots a_k \binom{n}{k}\frac{(n+2\beta+1)_k}{(2\alpha+\eta+\delta(k-1))_k}\bra k^\dagger|. 
\end{equation}
Regarding $W=X+\mu Z$, the eigenvalues are $\theta(x)=x\in \dC$ and the eigenvectors are given by
\begin{equation}\label{eq:Expression wx}
\ket{w_x^+} = w^+_0(x)\sum_{k=0}^\infty  \frac{(x -\delta \mu -\alpha -(1+\delta)(k-1))_k}{a_1\cdots a_k(\mu+\beta+1)_k}\ket{k}. 
\end{equation}
For $\lambda(x)=x\in \dC$, the generalized eigenvectors $|d^+_x\rangle$ of the GEVP \eqref{eq:Def GEVP} are explicitly given by
\begin{equation}\label{eq:Expression dx+}
\ket{d^+_x} = d^+_0(x) \sum_{k\geq 0}(-1-2\delta)^k\frac{(-\delta x+(1+2\delta)\alpha)_k}{a_1\cdots a_k(x+\beta+1)_k} \ket{k}. 
\end{equation}

From now on, and for the rest of the paper, we choose the normalization constants to be $1$, i.e., $v_0^+(n)=w_0^+(x)=d_0^+(x)=1$. We now collect the actions of the operators on the different eigenbases that will be necessary for future use. We start with the actions of $X, Z$ on the vectors $\ket{v_n^+}$ is given by the following tridiagonal formulas
\begin{align}
    X\ket{v_n^+} & = X_{n-1,n}^{(v^+)}\ket{v_{n-1}^+}+X_{n,n}^{(v^+)}\ket{v_n^+}+X_{n+1,n}^{(v^+)}\ket{v_{n+1}^+},\\
    Z\ket{v_n^+} & = Z_{n-1,n}^{(v^+)}\ket{v_{n-1}^+}+Z_{n,n}^{(v^+)}\ket{v_n^+}+Z_{n+1,n}^{(v^+)}\ket{v_{n+1}^+},
\end{align}
with 
\begin{align*}
    X_{n-1,n}^{(v^+)}&=a_n(n+\beta),\\
    X_{n,n}^{(v^+)}&=-\frac{\eta}{2},\\
    X_{n+1,n}^{(v^+)}&=\frac{(n+1)(2\alpha+\eta+(1+2\delta)n)(2\beta+n+1)(2\alpha+\eta-(1+2\delta)(2\beta+2))}{2a_{n+1}(2\beta+2n+1)(2\beta+2n+2)(2\beta+2n+3)},\\
    \\
    Z_{n-1,n}^{(v^+)}&= a_n,\\
    Z_{n,n}^{(v^+)} &  = \delta-\frac{(n+1)(2\alpha+\eta+(1+2\delta)n)}{2(\beta+n+1)}+\frac{n(2\alpha+\eta+(1+2\delta)(n-1))}{2(\beta+n)},\\
    Z_{n+1,n}^{(v^+)}& = -\frac{X_{n,n+1}}{\beta+n+1}=-\frac{(n+1)(2\alpha+\eta+(1+2\delta)n)(2\beta+n+1)(2\alpha+\eta-(1+2\delta)(2\beta+n+2))}{4a_{n+1}(2\beta+2n+1)(\beta+n+1)^2(2\beta+2n+3)}.
\end{align*}
This leads to the following action of $W=X+\mu Z$
\begin{equation}
    W\ket{v_n^+}=W_{n-1,n}^{(v^+)}\ket{v_{n-1}^+}+W_{n,n}^{(v^+)}\ket{v_n^+}+W_{n+1,n}^{(v^+)}\ket{v_{n+1}^+},
\end{equation}
where
\begin{align}
    W_{n-1,n}^{(v^+)}&=a_n(n+\beta+\mu),\\
    W_{n,n}^{(v^+)}&=-\frac{\eta}{2}+\mu\left( \delta-\frac{(n+1)(2\alpha+\eta+(1+2\delta)n)}{2(\beta+n+1)}+\frac{n(2\alpha+\eta+(1+2\delta)n+1)}{2(\beta+n)}\right),\\
    W_{n+1,n}^{(v^+)}&=\frac{(n+1)(\beta-\mu+n+1)(2\alpha+\eta+(1+2\delta)n)(2\beta+n+1)(2\alpha+2\beta+\eta-(1+2\delta)(n+2))}{4a_{n+1}(2\beta+2n+1)(\beta+n+1)^2(2\beta+2n+3)}.
\end{align}
The action of $V$ on the vectors $\ket{w_x^+}$ can be expressed as follows
\begin{equation}\label{eq:tridiagonalAction V positive case}
V\ket{w_x^+}=V_{x-1,x}^{(w^+)}\ket{w_{x-1}^+} +V_{x,x}^{(w^+)}\ket{w_x^+}+ V_{x+1,x}^{(w^+)}\ket{w_{x+1}^+},
\end{equation}
with
\begin{align*}
    V_{x-1,x}^{(w^+)}&=(\alpha-\delta(\beta+1)-x)(\alpha+\eta+\mu+x-(1+\delta)(\beta+1)),\\
    V_{x,x}^{(w^+)}&=-V_{x,x-1}-V_{x,x+1}-\beta(\beta+1),\\
    V_{x+1,x}^{(w^+)}&=-(\mu-\alpha+x-(1+\delta)(\beta+1))(\alpha+\beta+\eta+x-\delta(\beta+1)).
\end{align*}
The action of $Z$ and $VZ$ on $\ket{d_x^+}$ are all tridiagonal with the following explicit expressions
\begin{align}
Z\ket{d_x^+}&=Z^{(d^+)}_{x,x} \ket{d_x^+}+ Z^{(d^+)}_{x+1,x}\ket{d_{x+1}^+},\label{eq:ZactionOn dx PositiveCase}\\
VZ\ket{d_x^+}&=VZ^{(d^+)}_{x-1,x}\ket{d^+_{x-1}}+VZ^{(d^+)}_{x,x} \ket{d_x^+}+ VZ^{(d^+)}_{x+1,x}\ket{d_{x+1}^+},\label{eq:VZactionOn dx PositiveCase}
\end{align}
with
\begin{align*}
    Z^{(d^+)}_{x,x}=-1,&\ \ 
    Z^{(d^+)}_{x+1,x}=-\frac{\alpha-x-(1+\delta)(\beta+1)}{\beta+x+1},
\end{align*}
and
\begin{align*}
    VZ^{(d^+)}_{x-1,x}&=\frac{x(x+1)(\alpha-x-(1+\delta)(\beta+1))}{\beta+x+1} ,\\
    VZ^{(d^+)}_{x,x}&=2\alpha\beta+\beta\eta+x\eta+2x^2-(1+2\delta)\beta(\beta+1),\\
    VZ^{(d^+)}_{x+1,x}&=-(\alpha+\eta+x-(1+\delta)(\beta+1))(\beta+x).
\end{align*}

\subsection{Representations of the negative type}
We now consider a representation of the negative type as described in Proposition \ref{prop:negativeType} and we find the following eigenvectors.

For $V$, the eigenvalues can be written $\rho(x)=(x-\beta)(\beta-x+1)$ with $x\in \dC$ and the eigenvectors are:
\begin{equation}\label{eq:ExpressionVx}
\ket{v_x^-} =v_0^-(x)\sum_{k=0}^\infty \frac{a_0\cdots a_{k-1}}{k!} \frac{(x-k+1)_k (x-2\beta-1)_k}{(2\alpha-\eta+\delta(k-1))_k} \ket{k}.
\end{equation}
For $W=X+\mu Z$, the eigenvalues can be written $\theta_n=-n(1+2\delta)-\alpha+\delta\mu$ and the eigenvectors are 
\begin{equation}\label{eq:ExpressionFn}
\ket{w_n^-}= w_0^-(n)\sum_{k=n}^\infty a_n\cdots a_{k-1}(-1-2\delta)^{k-n} \frac{(-\mu-\beta)_{k}}{(-\mu-\beta)_n (k - n)!} \ket{k} . 
\end{equation}
The family $\{\ket{w_n^-}\}_{n\in \dN}$ is a basis of $\dC^\dN$, and its dual basis $\bra {w_n^-}^\dagger|$ is then given by
\begin{equation}\label{eq:ExpressionFn*}
    \bra {w_n^-}^\dagger|= \frac{(-\mu-\beta)_n(1+2\delta)^na_0\cdots a_{n-1}}{n!w_0^-(n)}\sum_{k=0}^n \frac{(-(1+2\delta))^k(-n)_k}{a_0\cdots a_{k-1}(-\mu-\beta)_k}\bra k^\dagger|.
\end{equation}
Regarding the GEVP \eqref{eq:Def GEVP}, there are no non zero solutions when $\delta=0$, except if $\alpha=k$ for some $k\in \dN$. For sake of simplicity we exclude this situation from the analysis. For $\delta=-1$, the possible eigenvalues are $\lambda_n=\alpha-n$ and the generalized eigenvectors $\ket{d^-_x}$ by:
\begin{equation}\label{eq:ExpressionDn-}
\ket{d^-_n} = d_0^-(n)\sum_{k=n}^\infty a_n\cdots a_{k-1}\frac{(\alpha-\beta)_{k-n}}{(k-n)!}\ket{k}.
\end{equation}
The family $\{ \ket{d_n^-}\}_{n\in \dN}$ is a basis of $\dC^\dN$, however we will not be interested in its dual basis since its basis vector will not satisfy the GEVP $\bra v |X^\top=\lambda_n \bra v|Z^\top$. Instead, we will consider the vectors $\bra {d_n^-}^*|$ which are defined as the solutions of the GEVP 
\begin{equation}\label{eq:TransposeGEVP}
    \bra {d_n^-}^*|X^\top=\lambda_n\bra {d_n^-}^*|Z^\top.
\end{equation}
Solving this equation leads to the following expression for $\bra {d_n^-}^*|$
\begin{equation}\label{eq:dualityd_n}
    \bra {d_n^-}^*|=d_0^*(n)\sum_{k=0}^n\frac{(-n)_k}{a_0\cdots a_{k-1}(\alpha-\beta-n)_k}\bra k^\dagger|.
\end{equation}
The family  $\bra {d_n^-}^*|$ is in fact the dual basis of the family $Z\ket{d_n^-}$, in the sense that 
\begin{equation}\label{eq:InnerProductGEVP}
    \bra {d_n^-}^*|Z\ket{d_m^-}=\kappa_n \delta_{nm},
\end{equation}
with $\kappa_n=-\frac{n!}{(1+\beta-\alpha)_n}$. This fact can be checked through a direct computation but was proven abstractly in \cite{TVZ2025MetaHahn}. We normalize $\bra {d_n^-}^*|$ so that $\kappa_n=1$ leading to 
\begin{equation}\label{eq:ExpressionDn-*}
    \bra {d_n^-}^*|=\frac{a_{n-1}\cdots a_0(-1)^{n+1}(\alpha-\beta-n)_n}{n!d_0^-(n)}\sum_{k=0}^n\frac{(-n)_k}{a_0\cdots a_{k-1}(\alpha-\beta-n)_k}\bra k^\dagger|.
\end{equation}

As for the positive type representations, we will choose the normalization constants to be $1$, i.e., $v_0^-(x)=w_0^-(n)=d_0^-(n)=1$. We describe now the actions needed in the rest of the paper, starting with the actions of $Z$ and $X$ on the vectors $\ket{v_x^-}$. 
\begin{align}
    Z\ket{v_x^-} &=Z_{x-1,x}^{(v^-)}\ket{v_{x-1}^-} +Z_{x,x}^{(v^-)} \ket{v_x^-} +Z_{x+1,x}^{(v^-)} \ket{v_{x+1}^-},\label{eq:Zaction on vx}\\
    X\ket{v_x^-} &=X_{x-1,x}^{(v^-)}\ket{v_{x-1}^-} +X_{x,x}^{(v^-)} \ket{v_x^-} +X_{x+1,x}^{(v^-)} \ket{v_{x+1}^-},\label{eq:Xaction on vx}
\end{align}
with 
\begin{align*}
    Z_{x-1,x}^{(v^-)}&=- \frac{x(2\alpha-\eta+(1+2\delta)(2\beta -x+1))}{(2\beta-2x+1)(2\beta-2x+2)},\\
    Z_{x,x}^{(v^-)}&= \delta - Z_{x-1,x}-Z_{x+1,x},\\
    Z_{x+1,x}^{(v^-)}&= -\frac{(2\alpha-\eta +(1+2\delta)x)(2\beta-x+1)}{(2\beta-2x)(2\beta-2x+1)},
\end{align*}
and
\begin{align*}
    X_{x-1,x}^{(v^-)}&= \frac{x(2\alpha-\eta+(1+2\delta)(2\beta -x+1))}{2(2\beta-2x+1)},\\
    X_{x,x}^{(v^-)}&= -\alpha - X_{x-1,x}-X_{x+1,x},\\
    X_{x+1,x}^{(v^-)}&= -\frac{(2\alpha-\eta +(1+2\delta)x)(2\beta-x+1)}{2(2\beta-2x+1)}.
\end{align*}
These two equations lead to the following action of $W=X+\mu Z$ on $\ket{v_x^-}$ 
\begin{equation}\label{eq:WactionOnvx}
    W\ket{v_x^-} =W_{x-1,x}^{(v^-)}\ket{v_{x-1}^-} +W_{x,x}^{(v^-)} \ket{v_x^-}+W_{x+1,x}^{(v^-)} \ket{v_{x+1}^-},
\end{equation}
with 
\begin{align*}
    W_{x-1,x}^{(v^-)}&= \frac{x(2\alpha-\eta+(1+2\delta)(2\beta -x+1))(\beta - \mu -x+1)}{(2\beta-2x+1)(2\beta-2x+2)},\\
    W_{x,x}^{(v^-)}&= \delta\mu -\alpha - W_{x-1,x}-W_{x+1,x},\\
    W_{x+1,x}^{(v^-)}&= -\frac{(2\alpha-\eta +(1+2\delta)x)(2\beta-x+1)(\beta+\mu-x)}{(2\beta-2x)(2\beta-2x+1)}. 
\end{align*}
Let us now describe the action of $V$ on the basis $\ket{w_n^-}$
\begin{equation}\label{eq:VactionOnw_n}
V\ket{w_n^-}=V_{n-1,n}^{(w^-)} \ket{w_{n-1}^-} +V_{n,n}^{(w^-)} \ket{w_n^-}+ V_{n+1,n}^{(w^-)}\ket{w_{n+1}^-},
\end{equation}
with
\begin{align*}
    V_{n-1,n}^{(w^-)}&=-\frac{n(2\alpha-\eta+(n-1)(1+2\delta))}{a_{n-1}} ,\\
    V_{n,n}^{(w^-)}&= -(\beta^2+\beta+2\mu n-2n^2)-(1+2\delta)(2\alpha-\eta)(\beta+\mu-2n) ,\\
    V_{n+1,n}^{(w^-)}&= a_n((1+2\delta)(\beta-\mu+n+1)+2\alpha-\eta)(\beta+\mu -n).
\end{align*}
We conclude this section with the actions of $Z$ and $ZV$ on the basis $\ket{d_n^-}$ explicitly given by
\begin{align}
    Z\ket{d_n^-}&=-\ket{d_n^-}+a_n\ket{d_{n+1}^-},\label{eq:ZactionOn dn}\\
    ZV\ket{d_n^-}&=ZV_{n-1,n}^{(d^-)}\ket{d_{n-1}^-}+ZV_{n,n}^{(d^-)}\ket{d_n^-}+ZV_{n+1,n}^{(d^-)}\ket{d_{n+1}^-}\label{eq:ZVactionOn dn},
\end{align}
with
\begin{align*}
    ZV_{n-1,n}^{(d^-)}&= \frac{n(2\alpha-\eta-n+1)}{a_{n-1}},\\
    ZV_{n,n}^{(d^-)}&= 2n^2+\eta n- 4\alpha n+2\alpha^2 - 2\alpha\beta + \beta^2 - \alpha\eta + \beta\eta  - 2\alpha + \beta+ \eta,\\
    ZV_{n+1,n}^{(d^-)}&= -a_n(\alpha-n+1)(\alpha-n).
\end{align*} 
 
%%%%%%%%%%%%%%%%%%%%%%
%%%%%%%%%%%%%%%%%%%%%%
%%%%%%%%%%%%%%%%%%%%%%
%%%%%%%%%%%%%%%%%%%%%%

\section{Continuous Hahn polynomials}\label{sec:continuousHahn}

The continuous Hahn polynomials $p_n(x)=p_n(x;a,b,c,d)$ are defined in terms of an hypergeometric function as follows
\begin{equation}\label{eq:DefinitionContinuousHahn}
    p_n(x;a,b,c,d)=i^n\frac{(a+c)_n(a+d)_n}{n!}{}_3F_2\left(\begin{array}{c} -n,n+a+b+c+d-1,a+ix\\ a+c,a+d\end{array} ;1\right).
\end{equation}
We refer to \cite{koekoek2010hypergeometric} for a description of their properties. 

In this section, we will show how one can recover continuous Hahn polynomials using representations of the positive type. More precisely, they relate to the following transition coefficient
\begin{equation}
    S_n(x)=\bra{{v_n^+}^\dagger}\ket{w_x^+}.
\end{equation}
That being done we will recover the recurrence and difference equations for the continuous Hahn polynomials using our knowledge about the representations of $m\mathfrak{H}$.

\subsection{Identification}
The next proposition makes precise the identification between the transition coefficients $S_n(x)$ and the continuous Hahn polynomials. 
\begin{proposition}\label{prop:IndentificationContinuousHahn}
    The function $S_n(x)$ can be identified with continuous Hahn polynomials as follows
    \begin{equation}\label{eq:IdentificationContinuousHahn}
        S_n(x)=\frac{(-1-2\delta)^ni^n n!}{a_1\cdots a_n(n+a+b+c+d-1)_n(a+c)_n}p_n((1+2\delta )ix;a,b,c,d),
    \end{equation}
    with $a=-\delta\mu+(1+2\delta)\alpha$, $b=\beta+1-(1+\delta)\mu-(1+2\delta)(\alpha+\eta)$, $c=\beta+1+(1+\delta)\mu-(1+2\delta)\alpha$, and $d=(1+2\delta)(\alpha+\eta)+\delta \mu$.
\end{proposition}
\begin{proof}
This is a consequence of the explicit expressions \eqref{eq:Expression en*} and \eqref{eq:Expression wx} for $\bra{{v_n^+}^\dagger}|$ and $\ket{w_x^+}$. Indeed
\begin{align}
    \bra{{v_n^+}^\dagger}\ket{w_x^+}&=\frac{(2\alpha+\eta+\delta(n-1))_n}{a_1\cdots a_n(n+2\beta+1)_n}\\&\times\sum_{k=0}^n\sum_{\ell\geq 0} 
    \ \binom{n}{k}\frac{a_1\cdots a_k(n+2\beta+1)_k(x-\delta\mu-\alpha-(1+\delta)(\ell-1))_\ell}{a_1\cdots a_\ell(2\alpha+\eta+\delta(k-1))_k(\mu+\beta+1)_\ell}\langle k^\dagger\ket{\ell}.
\end{align}
Using $\langle k^\dagger\ket{\ell}=\delta_{k,\ell}$, one gets the result by looking separately at the two values $0$ and $-1$ for the representation parameter $\delta$ and the definition of $a,\ b,\ c$ and $d$ given in the Proposition.
\end{proof}

\subsection{Recurrence relation}
We recover the three terms recurrence relation of the continuous Hahn polynomials via 
\begin{multline}
x\bra {v_n^+}^\dagger\ket{w^+_x}=\bra{{v_n^+}^\dagger}|W\ket{w_x^+}
=\left(\bra{{v_n^+}^\dagger}|W^\top\right)\ket{w_x^+}\\= W_{n,n-1}^{(v+)}\bra{{v_{n-1}^+}^\dagger}\ket{w_x^+} +W_{n,n}^{(v+)} \bra{{v_n^+}^\dagger}\ket{w_x^+} +W_{n,n+1} ^{(v+)}\bra{{v_{n+1}^+}^\dagger}\ket{w_x^+},
\end{multline}
which is equivalent to 
\begin{equation}
xS_n(x)=W_{n,n-1}^{(v+)} S_{n-1}(x)+W_{n,n}^{(v+)}S_n(x)+W_{n,n+1}^{(v+)}S_{n+1}(x). 
\end{equation}

Using the normalized continuous Hahn polynomials $\tilde{p}_n(x)=\frac{n!}{i^n(a+c)_n (a+d)_n}p_n(x)$, we obtain the three terms recurrence relation 
\begin{multline}
    (a+ix)\tilde{p}_n(x)=-\frac{a_n(2n+a+b+c+d-2)(2n+a+b+c+d-3)}{(a+d+n-1)(n+a+b+c+d-2)}W_{n,n-1}^{(v+)}\tilde{p}_{n-1}(x)\\
    +(W_{n,n}^{(v+)}-a)\tilde{p}_n(x)-\frac{(a+d+n)(n+a+b+c+d-1)}{a_{n+1}(2n+a+b+c+d-1)(2n+a+b+c+d)}  W_{n,n+1}^{(v+)} \tilde{p}_{n+1}(x).
\end{multline}
A direct computation shows that it is equivalent to the classical  three terms relation (see \cite{koekoek2010hypergeometric})
\begin{equation}
    (a+ix)\tilde{p}_n(x)=A_n\tilde{p}_{n+1}(x)-\left[ A_n+C_n\right]\tilde{p}_n(x)+ C_n\tilde{p}_{n-1}(x),
\end{equation}
with
\begin{align*}
    A_n&=-\frac{(n+a+b+c+d-1)(n+a+c)(n+a+d)}{(2n+a+b+c+d-1)(2n+a+b+c+d)},\\
    C_n&= \frac{n(n+b+c-1)(n+b+d-1)}{(2n+a+b+c+d-2)(2n+a+b+c+d-1)}.
\end{align*}

\subsection{Difference equation}
We recover the difference equation of the continuous Hahn polynomials through the formula
\begin{multline}
\rho_n\bra{{v_n^+}^\dagger}\ket{w^+_x}=\left(\bra{{v_n^+}^\dagger}|V^T\right)\ket{w^+_x}=\bra{{v_n^+}^\dagger}|V\ket{w^+_x} 
\\=V_{x,x-1}^{(w^+)} \bra{{v_n^+}^\dagger}\ket{w_{x-1}^+} +V_{x,x}^{(w^+)} \bra{{v_n^+}^\dagger}\ket{w_x^+}+ V_{x,x+1}^{(w^+)}\bra{{v_n^+}^\dagger}\ket{w_{x+1}^+},
\end{multline}
which is equivalent to 
\begin{equation}
n(n+2\beta+1) S_n(x)=-V_{x-1,x}^{(w^+)}S_n(x-1)+[V_{x-1,x}^{(w^+)}+V_{x+1,x}^{(w^+)}]S_n(x)-V_{x+1,x}^{(w^+)}S_n(x+1).
\end{equation}
Using the identification \eqref{eq:IdentificationContinuousHahn} this allows to recover the difference equation for the continuous Hahn polynomials as found in \cite{koekoek2010hypergeometric}
\begin{equation}
  n(n+a+b+c+d-1)p_n(x)=A(x)p_n(x+i)-\left[ A(x)+B(x)\right]p_n(x) + B(x)p_n(x-i),
\end{equation}
with 
\begin{equation}
    A(x)=(c-ix)(d-ix), \qquad    B(x)=(a+ix)(b+ix) .
\end{equation}

%%%%%%%%%%%%%%%%%%%%%%
%%%%%%%%%%%%%%%%%%%%%%
%%%%%%%%%%%%%%%%%%%%%%
%%%%%%%%%%%%%%%%%%%%%%

\section{Continuous dual Hahn polynomials}\label{sec:ContinuousDualHahn}
The continuous dual Hahn polynomials $q_n(x^2)=q_n(x^2,a,b,c)$ are defined as follows:
\begin{equation}
q_n(x^2;a,b,c)= (a+b)_n(a+c)_n ~{}_3F_2\left(\begin{array}{c} -n,a+ix,a-ix\\ a+b,a+c\end{array} ;1\right). 
\end{equation}
We also refer to \cite{koekoek2010hypergeometric} for a description of their properties.

In this section, we will show how one can recover continuous dual Hahn polynomials using representations of the negative type. More precisely, they will be related to the following transition coefficient
\begin{equation}
    R_n(x)=\bra{{w_n^-}^\dagger}\ket{v_x^-}.
\end{equation}
This is the analogue of the previous section applied to the negative type representations. 

%%%%%%%%%%%%%%%%%%%%%%
%%%%%%%%%%%%%%%%%%%%%%
%%%%%%%%%%%%%%%%%%%%%%
%%%%%%%%%%%%%%%%%%%%%%

\subsection{Identification}
Based on the explicit expression \eqref{eq:ExpressionFn*} of $\bra{{w_n^-}^\dagger}|$ and \eqref{eq:ExpressionVx} of $\ket{v^-_x}$, we find the identification between $R_n(x)$ and the continuous dual Hahn polynomials exactly as in Proposition \ref{prop:IndentificationContinuousHahn}. This is given in the following Proposition. 

\begin{proposition}\label{prop:IdentificationDualHahn}
For $n\in \dN$ we have:
    \begin{equation}
    R_n(x)=\frac{(1+2\delta)^na_0\cdots a_{n-1}}{n!(a+b)_n}q_n(y^2;a,b,c)),
    \end{equation}
    with $iy=x+a$, $a=-\frac{2\beta+1}{2}$, $a+b=(1+2\delta)(2\alpha-\eta)$ and $a+c=-\mu-\beta$. 
\end{proposition}
\subsection{Recurrence relation}
From the tri-diagonal action \eqref{eq:VactionOnw_n} of $V$ on $\ket{w_n^-}$, we get the following tri-diagonal action of $V^\top$ on $\bra{{w_n^-}^\dagger}|$ 
\begin{equation}
\bra{{w_n^-}^\dagger}|V^\top=V_{n,n+1}^{(w^-)}\bra{{w_{n+1}^-}^\dagger}|+V_{n,n}^{(w^-)} \bra{{w_n^-}^\dagger}|+V_{n,n-1}^{(w^-)}\bra{{w_{n-1}^-}^\dagger}|.
\end{equation}
Similarly as for positive type representations we get
\begin{equation}
\mu(x) R_n(x)= V_{n,n+1}^{(w^-)}R_{n+1}(x)+V_{n,n}^{(w^-)} R_n(x)+V_{n,n-1}^{(w^-)}R_{n-1}(x). 
\end{equation}
Using Proposition \ref{prop:IdentificationDualHahn}, we recover the normalized recurrence relation for the continuous dual Hahn polynomials
\begin{equation}
-(a^2+y^2)\tilde{q}_n(y^2)=\tilde{q}_{n+1}(y^2)-\left[A_n+C_n\right]\tilde{q}_n(y^2)+A_{n-1}C_n \tilde{q}_{n-1}(y^2),
\end{equation}
with
\begin{equation}
A_n=(n+a+b)(n+a+c),\qquad C_n=n(n+b+c-1).
\end{equation}

\subsection{Difference equation}
We recover the difference equation for the continuous dual Hahn using the following formula, found using the action \eqref{eq:WactionOnvx} of $W$ on $\ket{v_x^-}$
\begin{equation}
\lambda_n R_n(x) =W_{x-1,x}^{(v^-)} R_n(x-1) -\left[ W_{x-1,x}^{(v^-)}+W_{x+1,x}^{(v^-)} \right] +W_{x+1,x}^{(v^-)}R_n(x+1).
\end{equation}
Together with Proposition \ref{prop:IdentificationDualHahn} this allows to recover the difference equation for the continuous dual Hahn polynomials as found in the literature \cite{koekoek2010hypergeometric}
\begin{equation}
nq_n(y^2)=A(y)q_n((y+i)^2)-\left[ A(y)+B(y)\right]q_n(y^2) +B(y)q_n((y-i)^2),
\end{equation}
with 
\begin{equation}
A(y)=\frac{(a-iy)(b-iy)(c-iy)}{2iy(2iy+1)},\qquad B(y)=\frac{(a+iy)(b+iy)(c+iy)}{2iy(2iy-1)}.
\end{equation}

\section{Biorthogonal rational functions associated to the positive type}\label{sec:BRFpositive}
Let us consider a representation of the positive type of $m\mathfrak{H}$ and considrer the transition coefficient
\begin{equation}
    U_n(x)=\bra{{v_n^+}^\dagger}\ket{d_x^+}.
\end{equation}
Through this section we will show that $U_n(x)$ can be expressed in terms of a ${}_3F_2$ hypergeometric function and study its properties, i.e., a biorthogonality relation, a recurrence and a difference equation.

\subsection{Identification}
Introduce the rational functions
\begin{equation}
    \mathcal{U}_n^{-1}(x;a,b,c)={}_3F_2\argu{-n,2b+n+1,x-a}{c,x+b+1}{1}.
\end{equation}
and
\begin{equation}
    \mathcal{U}_n^{0}(x;a,b,c)={}_3F_2\argu{-n,2b+n+1,a}{c,x+b+1}{1}.
\end{equation}

To ensure compactness of the results, we combine both functions into one using the following notation
\begin{equation}
    \mathcal{U}^\varepsilon_n(x;a,b,c)={}_3F_2\argu{-n,2b+n+1,-\varepsilon x+(1+2\varepsilon)a}{c,x+b+1}{1},
\end{equation}
with $\varepsilon\in \{0,-1\}$. We will sometimes forgot the parameters and only write $\mathcal{U}_n^\varepsilon(x)$.

The next proposition describes the transition coefficient $U_n(x)$ in terms of the rational function $\mathcal{U}_n^\varepsilon(x)$.

\begin{proposition}\label{prop:IdentificationPositive}
    The transition coefficients $U_n(x)$ are explicitly related to the rational function $\mathcal{U}_n^\varepsilon(x)$ by the following expression
        \begin{equation}
        U_n(x)=\frac{((1+2\varepsilon)c+\varepsilon(n-1))_n}{a_1\cdots a_n(2b+n+1)_n}\mathcal{U}_n^\varepsilon(x;a,b,c),
    \end{equation}
    where $a=\alpha$, $b=\beta$, $c=(1+2\delta)(2\alpha+\eta)$, and $\varepsilon=\delta$.
\end{proposition}
\begin{proof}
    The proof is done similarly as in Proposition \ref{prop:IndentificationContinuousHahn} using the explicit expressions \eqref{eq:Expression en*} for $\bra {v_n^+}^\dagger|$ and \eqref{eq:Expression dx+} for $\ket{d^+_x}$. 
\end{proof}

\subsection{Recurrence relation}
The rational function $\mathcal{U}_n^\varepsilon$ turns out to satisfy a recurrence relation as described in the following proposition.
\begin{proposition}
The rational functions $\mathcal{U}_n^\varepsilon(x;a,b,c)$
satisfy the following three term recurrence relation
\begin{multline}\label{eq:Recurrence for U}
    x\left(A_n\mathcal{U}_{n-1}^\varepsilon(x;a,b,c)+B_n\mathcal{U}_n^\varepsilon(x;a,b,c)+C_n\mathcal{U}_{n+1}^\varepsilon(x;a,b,c)\right)\\
    =\left(A'_n\mathcal{U}_{n-1}^\varepsilon(x;a,b,c)+B'_n\mathcal{U}_n^\varepsilon(x;a,b,c)+C'_n\mathcal{U}_{n+1}^\varepsilon(x;a,b,c)\right),
\end{multline}
where 
\begin{align*}
    A_n&= -\frac{n(1+2\varepsilon)(c-n-1)}{2(b+n)(2b+2n+1)},\\
    B_n&= \varepsilon-\frac{(n+1)(1+2\varepsilon)(c+n)}{2(b+n+1)}+\frac{n(1+2\varepsilon)(c+n-1)}{2(b+n)},\\
    C_n&= \frac{2b+n+1}{2(1+2\varepsilon)(c+n-1)(b+n+1)(2b+2n+1)},\\
    A'_n&= \frac{n(1+2\varepsilon)(c-n-1)}{2(2b+2n+1)},\\
    B'_n&=-(1+2\varepsilon)\frac{c}{2} +a,\\
    C'_n&= \frac{2b+n+1}{2(1+2\varepsilon)(c+n-1)(2b+2n+1)}.\\
\end{align*}
\end{proposition}
\begin{proof}
    The recurrence relation is obtained from the corresponding relation on the transition coefficients $U_n(x)=\bra {v_n^+}^\dagger\ket{d_x^+}$ which is a consequence of the GEVP \eqref{eq:Def GEVP} 
    \begin{equation}
        \lambda_x \bra {v_n^+}^\dagger|~X\ket{d_x^+} =\bra {v_n^+}^\dagger|~Z\ket{d_x^+}.
    \end{equation}
    The proof is then a consequence of the explicit actions of $X$ and $Z$ on the vectors $\bra {v_n^+}^\dagger|$ combined with the identification proved in Proposition \ref{prop:IdentificationPositive}.
\end{proof}
\subsection{Difference equation} 
The functions $\mathcal{U}_n^\varepsilon$ also satisfy a difference equation which complete the description of their bispectral property.
\begin{proposition}
The rational functions $\mathcal{U}_n^\varepsilon(x;a,b,c)$
satisfy the following difference equation
\begin{multline}\label{eq:DifferenceEquationUn}
    (n+b)(n+b+1)\left( \mathcal{U}_n^\varepsilon(x;a,b,c)+\frac{a-x-(1+\varepsilon)(b+1)}{b+x+1}\mathcal{U}_n^\varepsilon(x+1;a,b,c)\right)\\
    =U_{x-1}\mathcal{U}_n^\varepsilon(x-1;a,b,c)+U_x\mathcal{U}_n^\varepsilon(x;a,b,c)+U_{x+1}\mathcal{U}_n^\varepsilon(x+1;a,b,c),
\end{multline}
where
\begin{align*}
    U_{x-1}&=\frac{x(x+1)(a-x-(1+\varepsilon)(b+1)}{b+x+1},\\ 
    U_x&=(1+2\varepsilon)b(c-b-1)+2x^2+x((1+2\varepsilon)c-2a),\\
    U_{x+1}&= -((1+2\varepsilon)c-a+x-(1+\varepsilon)(b+1))(b+x) .
\end{align*}
\end{proposition}
\begin{proof}
    Recalling that $\langle {v_n^+}^\dagger|$ is an eigenvector of $V^\top$ with eigenvalue $\rho_n=-(n+\beta)(n+\beta+1)$, we have the following 
    \begin{equation}
        \rho_n\bra {v_n^+}^\dagger|Z\ket{\ d^+_x}=\left(\bra {v_n^+}^\dagger|V^\top \right)|Z\ket{\ d^+_x}=\bra{v_n^+}^\dagger|VZ \ket{\ d^+_x}. 
    \end{equation}
    Using the action of $VZ$ and $Z$ on $\ket{d_x^+}$ described in \eqref{eq:VZactionOn dx PositiveCase} and \eqref{eq:ZactionOn dx PositiveCase} this is equivalent to 
    \begin{multline}
        \rho_n\left( Z_{x,x}^{(d^+)})\bra {v_n^+}^\dagger\ket{\ d^+_x}+Z_{x+1,x}^{(d^+)})\bra {v_n^+}^\dagger\ket{\ d^+_{x+1}}\right)=\\
        VZ_{x-1,x}^{(d^+)}\bra {v_n^+}^\dagger\ket{\ d^+_{x-1}}+VZ_{x,x}^{(d^+)}\bra {v_n^+}^\dagger\ket{\ d^+_{x}}+VZ_{x+1,x}^{(d^+)}\bra {v_n^+}^\dagger\ket{\ d^+_{x+1}}.
    \end{multline}
    Looking at the explicit expressions for the coefficients and using the identification of parameters $a=\alpha$, $b=\beta$, $c=(1+\varepsilon)(2\alpha+\eta)$ and $\varepsilon=\delta$ lead to the result. 
\end{proof}

\subsection{Biorthogonality relation}
We conclude the exploration of the properties of the functions $\mathcal{U}_n^\varepsilon$ with the description of their biorthogonal partners and the biorthogonality relation. We first recall the classical family of Jacobi polynomials, denoted $P_n^{(\alpha,\beta)}$, which can be define as the family of orthogonal polynomials satisfying the following orthogonality relation for $\alpha>-1,\beta>-1$
\begin{equation}\label{eq:OrthogonalityJacobi}
    \int_0^1 P_n^{(\alpha,\beta)}(1-2u)   P_k^{(\alpha,\beta)}(1-2u) u^\alpha(1-u)^\beta~du=\frac{\Gamma(n+\alpha+1)\Gamma(n+\beta+1)}{n!(2n+\alpha+\beta+1)\Gamma(n+\alpha+\beta+1)}\delta_{n,k}.
\end{equation}
They satisfy the following remarkable identity \cite{GradshteinRyshik15}, valid for $Re(\lambda)>0$ and $Re(\mu)>0$, which gives an integral representation for ${}_3F_2$ hypergeometric functions 
\begin{equation}\label{eq:IntegralRep3F2}
    \int_0^1u^{\lambda-1}(1-u)^{\mu-1}P_n^{(\alpha,\beta)}(1-\gamma u)~du=\frac{(\alpha+1)_nB(\lambda,\mu)}{n!}{}_3F_2 \argu{-n,n+\alpha+\beta+1,\lambda}{\alpha+1,\lambda+\mu}{\frac{1}{2}\gamma},
\end{equation}
where $B(x,y)$ denotes the usual Beta function.  Applied with $\gamma=2, \alpha=c-1,\beta=2b-c+1,\lambda=-\varepsilon x +(1+2\varepsilon)a$ and $\mu=(1+\varepsilon)x+b-(1+2\varepsilon)a+1$ this provides the following integral representation of our rational function $\mathcal{U}_n^\varepsilon(x)$
\begin{multline}\label{eq:IntegralUn}
    \mathcal{U}^\varepsilon_n(x;a,b,c)=\frac{n!}{(c)_nB(-\varepsilon x +(1+2\varepsilon)a,(1+\varepsilon)x+b-(1+2\varepsilon)a+1)}\\\times \int_0^1u^{-\varepsilon x +(1+2\varepsilon)a-1}(1-u)^{(1+\varepsilon)x+b-(1+2\varepsilon)a}P_n^{(c-1,2b-c+1)}(1-2 u)~du.
\end{multline}

Let us define the following rational function,which will correspond to the biorthogonal partner of the function $\mathcal{U}_n^\varepsilon(x)$
\begin{equation}
    \mathcal{V}_n^\varepsilon(x;a,b,c)={}_3F_2\argu{-n,n+2b+1,\varepsilon x+c-(1+2\varepsilon)a+1}{c,-x+b+3}{1}.
\end{equation}
The connection with the functions $\mathcal{U}_n$ given by
\begin{equation}
    \mathcal{V}_n^\varepsilon(x;a,b,c)=\mathcal{U}_n^\varepsilon(-x+1;(1+2\varepsilon)c-a+1,b,c),
\end{equation}
yields the following integral representation using formula \eqref{eq:IntegralUn}
\begin{multline}\label{eq:IntegralVn}
    \mathcal{V}_n^\varepsilon(x;a,b,c)=\frac{n!}{(c)_nB(\varepsilon x+c-(1+2\varepsilon)a+\varepsilon+1,-(1+\varepsilon)x+b-c+(1+2\varepsilon)a-\varepsilon+1)}\\
    \times \int_0^1 u^{\varepsilon x+c-(1+2\varepsilon)a+\varepsilon}(1-u)^{-(1+\varepsilon)x+b-c+(1+2\varepsilon)a-\varepsilon}P_n^{(c-1,2b-c+1)}(1-2u)~du.
\end{multline}

\begin{proposition}
    Assume the real part of 
    \begin{equation}
        (1+2\varepsilon)a,\ b-(1+2\varepsilon)a)+1,\ c,\  2b-c+1,\ \text{and } b-c+(1+2\varepsilon)a+1,
    \end{equation}
    are all positive. 
    The rational functions $\mathcal{U}^\varepsilon_n(x;a,b,c)$ and $\mathcal{V}^\varepsilon_n(x;a,b,c)$ satisfy the following biorthogonality relation 
    \begin{equation}
        \int_{-\infty}^{+\infty} \mathcal{U}_n^\varepsilon(ix;a,b,c)\mathcal{V}_k^\varepsilon(ix;a,b,c)w(x)~dx=\frac{2\pi n!\Gamma(c)\Gamma(n+2b-c+2)}{(c)_n(2n+2b+1)\Gamma(n+2b+1)}\delta_{n,k},
    \end{equation}
    where the weight $w(x)$ is given by
    \begin{multline}\label{eq:WeightUnVn}
            w(x)=B(-\varepsilon ix +(1+2\varepsilon)a,(1+\varepsilon)ix+b-(1+2\varepsilon)a+1)\\\times B(\varepsilon ix+c-(1+2\varepsilon)a+\varepsilon+1,-(1+\varepsilon)ix+b-c+(1+2\varepsilon)a-\varepsilon+1).
    \end{multline}
\end{proposition}
    
\begin{proof}
    Using the integral representations \eqref{eq:IntegralUn} and \eqref{eq:IntegralVn} together with \eqref{eq:WeightUnVn} we prove
    \begin{multline}
        \int_{-\infty}^{+\infty} \mathcal{U}_n^\varepsilon(ix;a,b,c)\mathcal{V}_k^\varepsilon(ix;a,b,c)w(x)~dx= \frac{n!k!}{(c)_n(c)_k}\\
       \times\int_0^1\int_0^1\left( \int_{-\infty}^{+\infty}\left(\frac{u}{v}\right)^{\varepsilon ix}\left( \frac{1-u}{1-v}\right)^{(1+\varepsilon) ix}~dx\right)u^{(1+2\varepsilon)a-1}(1-u)^{b-(1+2\varepsilon)a}P_n^{(c-1,b-c-1)}(1-2u)\\
        \times v^{c-(1+2\varepsilon)a+\varepsilon}(1-v)^{b-c+(1+2\varepsilon)a-\varepsilon}P_k^{(c-1,b-c-1)}(1-2v)~dudv
    \end{multline}
    The following formula holds true
    \begin{equation}
        \int_{-\infty}^{+\infty}\left( \frac{u}{v}\right)^{\varepsilon ix}\left( \frac{1-u}{1-v}\right)^{(1+\varepsilon) ix}~dx=2\pi u^{-\varepsilon}(1-u)^{1+\varepsilon}\delta (u-y).
    \end{equation}
    Indeed, for $\varepsilon =0$ we have
\begin{equation}
\int_{-\infty}^{+\infty}\left( \frac{1-u}{1-v}\right)^{-ix}~dx=\int_{-\infty}^{+\infty}e^{-ix \ln \left(\frac{1-u}{1-v}\right)}~dx=2\pi \delta \left(\ln\left(\frac{1-u}{1-v}\right)\right)=2\pi (1-u) \delta(u-v).
\end{equation}
The case $\varepsilon=-1$ is proved similarly. 
    This leads to 
    \begin{multline}
        \int_{-\infty}^{+\infty} \mathcal{U}_n^\varepsilon(ix;a,b,c)\mathcal{V}_k^\varepsilon(ix;a,b,c)w(x)~dx= \frac{2\pi n!k!}{(c)_n(c)_k}\\
        \times \int_0^1 u^{c-1}(1-u)^{2b-c+1}P_n^{(c-1,2b-c+1)}(1-2u)P_k^{(c-1,2b-c+1)}(1-2u)~du.
        \end{multline}
        We end the proof using the orthogonality relation \eqref{eq:OrthogonalityJacobi} for the Jacobi polynomials . 
\end{proof}
%%%%%%%%%%%%%%%%%%%%%%
%%%%%%%%%%%%%%%%%%%%%%
%%%%%%%%%%%%%%%%%%%%%%
%%%%%%%%%%%%%%%%%%%%%%

\section{Rational functions associated to the negative type}\label{sec:BRFnegative}
We now consider a representation of $m\mathfrak{H}$ of the negative type and the transition coefficient
\begin{equation}
    F_n(x)=\bra {d_n^-}^*\ket{v_x^-}.
\end{equation}

This section will be devoted to prove that it can also be written in terms of a ${}_3F_2$ hypergeometric function and to prove a recurrence, a difference equation and a biorthogonality relation satisfied by this function. Recall that, for negative type representations, solutions to the GEVP \eqref{eq:Def GEVP} are found for $\delta=-1$ only, which will thus be assume through this section.  

\subsection{Identification}
Let us introduce the following polynomial family
\begin{equation}
    \mathcal{F}_n(x;a,b,c)={}_3F_2\argu{-n,-x,x-a}{b-n,c}{1}.
\end{equation}
As for the functions $\mathcal{U}_n(x)$, we will only write $\mathcal{F}_n(x)$ when the confusion is not possible. 

We relate the transition coefficients $F_n(x)$ to the hypergeometric functions $\mathcal{F}_n(x)$ in the following proposition.
\begin{proposition}\label{prop:IdentificationNegative}
The transition coefficients $F_n(x)$ are expressed in terms of the functions $\mathcal{F}_n$ as follows
\begin{equation}
    F_n(x)=\frac{a_0\cdots a_{n-1}(-1)^{n+1}(b-n)_n}{n!}\mathcal{F}_n(x;a,b,c).
\end{equation}
with $a=2\beta+1$, $b=\alpha-\beta$, and $c=-2\alpha+\eta$.
\end{proposition}
\begin{proof}
    The proof is done as in Proposition \ref{prop:IdentificationDualHahn} using the explicit expressions for $\bra{{d_n^-}^*}|$ and $\ket{v_x}$ respectively given in \eqref{eq:ExpressionDn-*} and \eqref{eq:ExpressionVx}.
\end{proof}

\begin{remark}
    Notice that the polynomial $\mathcal{F}_n(x)$ is the dual of $\mathcal{U}_n^{-1}(x)$ in the sense that $n$ and $x$ are exchanged between the two functions. This remark is connected to the fact that the negative type representation can be understood as the dual of a representation of the positive type as mentioned in Remark \ref{rem:Duality}.
\end{remark}

\subsection{Recurrence relation}
Using the identification made in Proposition \ref{prop:IdentificationNegative}, we prove a recurrence relation for the functions $\mathcal{F}_n$.
\begin{proposition}
The rational functions $\mathcal{F}_n(x;a,b,c)$ satisfy the following three term recurrence relation 
\begin{multline}\label{eq:RecurrenceRelation Rational Functions F_n}
A_n\mathcal{F}_{n+1}(x;a,b,c)+B_n\mathcal{F}_n(x;a,b,c)_n+C_n\mathcal{F}_{n-1}(x;a,b,c) \\
=(x-\beta)(\beta-x+1)\left[-\mathcal{F}_n(x;a,b,c)+\frac{n}{n+\beta+\alpha}\mathcal{F}_{n-1}(x;a,b,c) \right],
\end{multline}
where 
\begin{align}
    A_n&=(2\alpha-\eta-n)(\beta-\alpha+n+1),\\
    B_n&= 2n^2+\eta n- 4\alpha n+2\alpha^2 - 2\alpha\beta + \beta^2 - \alpha\eta + \beta\eta  - 2\alpha + \beta+ \eta,\\
    C_n&=\frac{n(\alpha-n)(\alpha-n+1)}{\alpha-\beta-n}.
\end{align}
\end{proposition}
\begin{proof}
    Since $\ket{v_x^-}$ is an eigenvector for $V$ with eigenvalue $\rho_x$ the  following relation holds true
    \begin{equation}
        \rho(x)\bra {d_n^-}^*|~Z\ket{v_x^-}=\bra {d_n^-}^*|~ZV\ket{v_x^-}.
    \end{equation}
    Using the actions \eqref{eq:ZactionOn dn} and \eqref{eq:ZVactionOn dn}, one can obtain the actions of $Z^\dagger$ and $(ZV)^\dagger$ on $\bra {d_n^-}^*|$ via the duality relation \eqref{eq:dualityd_n}. These actions together with the identification made in Proposition \ref{prop:IdentificationNegative} lead to the desired result. 
\end{proof}

\subsection{Difference equation}
Similarly as for the family $\mathcal{U}^\varepsilon_n$, the functions $\mathcal{F}_n$ turn out to satisfy a generalized difference equation as in the following Proposition.
\begin{proposition}
    The rational functions $\mathcal{F}_n(x;a,b,c)$ satisfy the following difference equation
    \begin{multline}
        X_{x,x+1}\mathcal{F}_n(x+1)+X_{x,x}\mathcal{F}_n(x)+X_{x,x-1}\mathcal{F}_n(x-1)=\\
        (\alpha-n)\left[Z_{x,x+1}\mathcal{F}_n(x+1)+Z_{x,x}\mathcal{F}_n(x)+Z_{x,x-1}\mathcal{F}_n(x-1)\right],
    \end{multline}
    where $X^{(v^-)}_{x+1,x},X^{(v^-)}_{x,x},X^{(v^-)}_{x-1,x}$ and $Z^{(v-)}_{x+1,x},Z^{(v-)}_{x,x},Z^{(v-)}_{x-1,x}$ where defined in \eqref{eq:Zaction on vx} and \eqref{eq:Xaction on vx}.
\end{proposition}

\begin{proof}
    The result is obtained in a similar fashion as for the recurrence equation \eqref{eq:Recurrence for U} using the fact that $\bra {d_n^-}^*|$ satisfy the generalized eigenvalue problem
    \begin{equation}
    \bra {d_n^-}^*|X^\top=\rho_n \bra {d_n^-}^*|Z^\top. 
    \end{equation}
    Applying this equality to the vectors $\ket{ v_x^-}$ and using the actions of $Z$ and $X$ on the vectors $\ket{v_x^-}$ given in the equations \eqref{eq:Zaction on vx} and \eqref{eq:Xaction on vx}. 
\end{proof}
%%%%%%%%%%%%%%%%%%%%%%
%%%%%%%%%%%%%%%%%%%%%%
%%%%%%%%%%%%%%%%%%%%%%
%%%%%%%%%%%%%%%%%%%%%%
\subsection{Biorthogonality relation}\label{sec:BiorthogonalityNegative}

We now identify the biorthogonal partners of the polynomials $\mathcal{F}_n$ together with the corresponding orthogonality measure. In this section, we will write $\mathcal{F}_n(x)$ for the function $\mathcal{F}_n(x;2\beta+1,\alpha-\beta,-2\alpha+\eta)$ Recall from Proposition \ref{prop:IdentificationDualHahn} that the parameters of the continuous dual Hahn polynomials attached to the negative type representations are
\begin{equation}\label{eq:CDH-parameters}
a=-\beta-\tfrac{1}{2},\qquad b=\eta-2\alpha+\beta+\tfrac{1}{2},\qquad c=\tfrac{1}{2}-\mu,
\end{equation}
with $x=\beta+\tfrac12+iy$. Throughout this subsection we assume
\begin{equation}\label{eq:positivity-conditions}
\alpha>\tfrac12,\qquad \beta<-\tfrac12,\qquad \eta>2\alpha-\beta-\tfrac12,
\end{equation}
which make $a$, $b$ and $\alpha-\tfrac12$ positive. Define the weight function
\begin{equation}\label{eq:weight-F}
W(y)
=\left|\frac{\Gamma(x-2\beta-1)\,\Gamma(x+\eta-2\alpha)\,\Gamma(x+\alpha-\beta-1)}{\Gamma(2x-2\beta-1)}\right|^2 ,
\end{equation}
which is the continuous dual Hahn weight with parameters $(-\beta-\tfrac12,\eta-2\alpha+\beta+\tfrac{1}{2},\alpha-\frac12)$. Finally, introduce the rational functions
\begin{equation}\label{eq:def-G}
\mathcal{G}_m(x)=\frac{(\eta-2\alpha+2\beta+1-x)_m}{(x+\alpha-\beta-1-m)_m\,(\alpha+\beta-m-x)_m}\;
\hg{3}{2}\argu{-m,\;x-2\beta-1,\;x+\alpha-\beta-1-m}{\alpha-\beta-1-m,\;x+2\alpha-\eta-2\beta-m}{1}.
\end{equation}
Notice that  $\mathcal{G}_m$ has its poles at the purely imaginary points $y=\pm i(\alpha-\frac12-j)$, $j=1,\dots,m$.

It is convenient to record the following factorization. For $x=\beta+\tfrac12+iy$, set
\begin{equation}\label{eq:def-Theta}
\Theta_m(y)=(\eta-2\alpha+\beta+\tfrac{1}{2}-iy)_m\;\hg{3}{2}\argu{-m,\;x-2\beta-1,\;x+\alpha-\beta-1-m}{\alpha-\beta-1-m,\;x+2\alpha-\eta-2\beta-m}{1},
\end{equation}
which is an entire function of $y^2$ up to the factor $(\eta-2\alpha+\beta+\tfrac{1}{2}-iy)_m$. Let
\begin{equation}\label{eq:def-cm-Wm}
c_m=\alpha-\tfrac12-m,\qquad
W_m(y)=\left|\frac{\Gamma(-\beta-\tfrac12+iy)\Gamma(\eta-2\alpha+\beta+\tfrac{1}{2}+iy)\Gamma(c_m+iy)}{\Gamma(2iy)}\right|^{2},
\end{equation}
so that $W_m$ is again a continuous dual Hahn weight, with third parameter $c_m$ in place of $c_0=\alpha-\frac12$. Since
$\left|\Gamma(c_m+iy)\right|^{2}=\left|\Gamma(c_0+iy)\right|^{2}/\left[(x+\alpha-\beta-1-m)_m(\alpha+\beta-m-x)_m\right]$
on the line $x=\beta+\frac12+iy$, comparison with \eqref{eq:def-G} gives
\begin{equation}\label{eq:factorization-GW}
\mathcal{G}_m(x)\,W(y)=\Theta_m(y)\,W_m(y).
\end{equation}
The parameter $c_m$ decreases with $m$, and thus the biorthogonality involves mass points as soon as it becomes negative, precisely as for the continuous dual Hahn polynomials themselves. The relation reads as follows.
\begin{theorem}\label{thm:biorthogonality-F}
    Let $\alpha,\beta,\eta$ satisfy \eqref{eq:positivity-conditions} and let $n,m\in\dN$. Then
    \begin{equation}\label{eq:biorthogonality-F}
        \frac{1}{2\pi}\int_{0}^{\infty} \mathcal{F}_n(x)\,\Theta_m(y)\,W_m(y)\,dy
        \;+\;\sum_{j=0}^{J_m-1}\lambda^{(m)}_j\,\mathcal{F}_n(x_{m,j})\,\Theta_m(y_{m,j})=h_n\,\delta_{n,m},
    \end{equation}
    with $x=\beta+\frac12+iy$ in the integral,
    \begin{equation}\label{eq:hn-F}
        h_n=\frac{n!\,\Gamma(\eta-2\alpha)\Gamma(\eta-\alpha+\beta)\Gamma(\alpha-\beta-1-n)}{(\alpha-\beta-n)_n},
    \end{equation}
    and where the discrete part is carried by the $J_m=\max(0,\lceil -c_m\rceil)$ points
    \begin{equation}\label{eq:mass-points}
    y_{m,j}=i(c_m+j),\qquad x_{m,j}=\beta+1-\alpha+m-j, \qquad 0\le j<J_m,
    \end{equation}
    that is, by those $j\ge0$ with $c_m+j<0$, with masses
    \begin{equation}\label{eq:masses}
    \lambda^{(m)}_j=\frac{\Gamma(a-c_m)\Gamma(b-c_m)\Gamma(a+c_m)\Gamma(b+c_m)}{\Gamma(-2c_m)}\;
    \frac{(-1)^j\,(2c_m)_j\,(c_m+1)_j\,(a+c_m)_j\,(b+c_m)_j}{j!\,(c_m)_j\,(c_m-a+1)_j\,(c_m-b+1)_j},
    \end{equation}
    where $a= -\beta-\tfrac12$ and $b=\eta-2\alpha+\beta+\tfrac12$. 
\end{theorem}

In view of \eqref{eq:factorization-GW}, when $m<\alpha-\frac12$ the sum in \eqref{eq:biorthogonality-F} is empty and the relation takes the simpler form
\begin{equation}\label{eq:biorthogonality-F-simple}
\frac{1}{2\pi}\int_{0}^{\infty} \mathcal{F}_n\left(\beta+\tfrac12+iy\right)\mathcal{G}_m\left(\beta+\tfrac12+iy\right)W(y)\,dy=h_n\,\delta_{n,m},
\qquad 0\le m<\alpha-\tfrac12,
\end{equation}
with the single, $m$-independent weight $W$ of \eqref{eq:weight-F}. Equivalently, and in a form that covers all $m$ at once: the left-hand side of \eqref{eq:biorthogonality-F} is the analytic continuation in $c=\tfrac12-\mu$, from $c>0$ to $c=c_m$, of the continuous dual Hahn functional with parameters $(a,b,c)$ applied to $\mathcal{F}_n\Theta_m$; the terms \eqref{eq:masses} are the mass points that this continuation produces when $c$ crosses the origin.

Note that for $n=m=0$ the relation \eqref{eq:biorthogonality-F} reduces to the evaluation $\frac{1}{2\pi}\int_0^{\infty}W(y)dy=\Gamma(\eta-2\alpha)\Gamma(\alpha-\beta-1)\Gamma(\eta-\alpha+\beta)$ of the total mass of the continuous dual Hahn weight, as it should.

The proof of Theorem \ref{thm:biorthogonality-F} rests on the following identity
\begin{equation}\label{eq:resolutionIdentity}
    \sum_{k\geq 0} \ket{w_k^-}\bra{{w_k^-}^\dagger}|=\operatorname{Id},
\end{equation}
and the connection with the continuous dual Hahn polynomials proved in Proposition \ref{prop:IdentificationDualHahn}. Recall that we are in the case $\delta=-1$, so that 
\begin{equation}
   \ket{w_k^-}=\sum_{j=k}^\infty a_k\cdots a_{j-1}\frac{(-\mu-\beta)_j}{(-\mu-\beta)_k(j-k)!}\ket{j}, \quad \bra{{w_k^-}^\dagger}|=\frac{(-\mu-\beta)_k(-1)^k}{k!}\sum_{j=0}^k a_j\cdots a_{k-1} \frac{(-k)_j}{(-\mu-\beta)_j}\bra j^\dagger |. 
\end{equation}
The normalized continuous dual Hahn polynomials $\tilde{q}_n(x)$, with our parameters $(a,b,c)$ given in \eqref{eq:CDH-parameters}, are  
\begin{equation}\label{eq:normalized-R}
\tilde{q}_k(y^2)=\hg{3}{2}\argu{-k,-x,x-2\beta-1}{-\mu-\beta,\eta-2\alpha}{1},
\end{equation}
so that $\tilde{q}_k(y^2)=\frac{k!(-1)^k}{(a+c)_ka_0\cdots a_{k-1}}\bra{{w_k^-}^\dagger}\ket{v^-_x}$, The following measure on $(0,+\infty)$
\begin{equation}\label{eq:sigma-mu}
d\sigma_{\mu}(y)=\frac{1}{2\pi}\left|\frac{\Gamma(a+iy)\Gamma(b+iy)\Gamma(c+iy)}{\Gamma(2iy)}\right|^2 dy,
\end{equation}
with the parameters \eqref{eq:CDH-parameters}, is known  to give the orthogonality relation for the continuous dual Hahn polynomials \cite{koekoek2010hypergeometric}, i. e., 
\begin{equation}\label{eq:hhat}
\int_0^{\infty}\tilde{q}_k(y^2) \tilde{q}_l(y^2) \,d\sigma_\mu(y)=\tilde{h}_k\,\delta_{k,l},\ 
\text{with }\tilde{h}_k=\frac{k!\,\Gamma(k+a+b)\Gamma(k+a+c)\Gamma(k+b+c)}{\left[(a+b)_k(a+c)_k\right]^2}.
\end{equation}

We first start with a couple of technical lemmas in order to prove Theorem \ref{thm:biorthogonality-F}.

\begin{lemma}\label{lem:overlaps}
For $k,m\in\dN$,
\begin{equation}\label{eq:overlaps}
    \Omega_{k,m}:=\bra{ {w_k^-}^\dagger}|~Z\ket{d^-_m}=-a_m\cdots a_{k-1}\frac{(-1)^k(-k)_m\,(m+1-\mu-\alpha)_{k-m}}{k!}.
\end{equation}
\end{lemma}

\begin{proof}
A direct computation using \eqref{eq:ExpressionDn-} and \eqref{eq:normalized-R} together with the Chu--Vandermonde summation gives
\begin{equation}
\bra{{w_k^-}^\dagger}\ket{d^-_m}=a_m \cdots a_{k-1}\frac{(-1)^k(-k)_m\,(m-\mu-\alpha)_{k-m}}{k!},
\end{equation}
and the result follows directly using $Z\ket{d^-_m}=-\ket{d^-_m}+\ket{d^-_{m+1}}$. 
\end{proof}

\begin{lemma}\label{lem:G-closed-form}
Assume $\mathrm{Re}(\mu+\alpha)<1$, $\eta-\alpha+\beta>0$ and $0\le m<\alpha-\frac12$. Then the series
\begin{equation}\label{eq:def-Gm-series}
G_m(y):=\sum_{k\geq 0}\frac{k!(-1)^k}{a_m\cdots a_{k-1}(a+c)_k}\frac{\tilde{q}_k(y^2)}{\tilde{h}_k}\,\Omega_{k,m},\qquad x=\beta+\tfrac12+iy,
\end{equation}
converges and admits the closed form
\begin{multline}\label{eq:G-closed-form}
G_m(y)=\frac{(-1)^{m+1}}{\Gamma(\eta-2\alpha)\Gamma(\eta-\alpha+\beta)\Gamma(\alpha-\beta-1-m)}
\left|\frac{\Gamma(\alpha-\frac12-m+iy)}{\Gamma(c+iy)}\right|^{2} \\
\times (b-iy)_m\;\hg{3}{2}\argu{-m,\;a+iy,\;\alpha-\frac12-m+iy}{\alpha-\beta-1-m,\;1-b-m+iy}{1}.
\end{multline}
\end{lemma}
\begin{proof}
Insert \eqref{eq:overlaps} and \eqref{eq:hhat} in \eqref{eq:def-Gm-series} and use $(-\mu-\beta)_k=(a+c)_k$ to get
\begin{equation}\label{eq:Gm-step1}
G_m(y)=\frac{-1}{\Gamma(a+b)\Gamma(a+c)\Gamma(b+c)}\sum_{k\geq 0}\frac{(-k)_m(m+1-\mu-\alpha)_{k-m}\,(a+b)_k}{k!\,(b+c)_k}\,\hat R_k(x).
\end{equation}
Next, since $b+c=1+\beta-\mu+\eta-2\alpha=(1-\mu-\alpha)+(\eta-\alpha+\beta)$, the Euler Beta integral gives
\begin{equation}
\frac{(m+1-\mu-\alpha)_{k-m}}{(b+c)_k}=\frac{\Gamma(b+c)}{\Gamma(m+1-\mu-\alpha)\,\Gamma(\eta-\alpha+\beta)}\int_0^1 s^{k-\mu-\alpha}\,(1-s)^{\eta-\alpha+\beta-1}\,ds ,
\end{equation}
which converts the $k$-dependent quotient into the $k$-th power $s^{k}$ under the integral sign. Using furthermore $(-k)_m s^k=(-1)^m s^{m}\frac{d^m}{ds^m}s^{k}$ and the generating function of the continuous dual Hahn polynomials \cite{koekoek2010hypergeometric}
\begin{equation}\label{eq:GF-CDH}
\sum_{k\geq0}\frac{q_k(y^2)}{(a+c)_k\,k!}\,s^k=(1-s)^{-b+iy}\hg{2}{1}\argu{a+iy,\;c+iy}{a+c}{s},
\qquad q_k=(a+b)_k(a+c)_k\,\tilde{q}_k ,
\end{equation}
one obtains
\begin{multline}\label{eq:Gm-step2}
G_m(y)=\frac{(-1)^{m+1}}{\Gamma(a+b)\Gamma(a+c)\Gamma(m+1-\mu-\alpha)\Gamma(\eta-\alpha+\beta)}\\ \times\int_0^1 s^{m-\mu-\alpha}(1-s)^{\eta-\alpha+\beta-1}\,\frac{d^m}{ds^m}\Big[(1-s)^{-b+iy}\hg{2}{1}\argu{a+iy,\;c+iy}{a+c}{s}\Big]ds .
\end{multline}
The $m$-th derivative is expanded with the Leibniz rule, using $\frac{d^r}{ds^r}\,{}_2F_1(A,B;C;s)=\frac{(A)_r(B)_r}{(C)_r}{}_2F_1(A+r,B+r;C+r;s)$. Each resulting term is an Euler-type integral which is evaluated via \eqref{eq:IntegralRep3F2} type integrals. More precisely,
\begin{equation}
\int_0^1 s^{e-1}(1-s)^{f-1}\hg{2}{1}\argu{A,B}{C}{s}\,ds=B(e,f)\;\hg{3}{2}\argu{A,B,e}{C,e+f}{1},
\end{equation}
with $e=m+1-\mu-\alpha$ and $f=\alpha-\frac12-m+r+iy$ in the term labeled by $r$ ($0\le r\le m$). The key observation is that in each term the parameter $e+f=c+r+iy$ coincides with the numerator parameter $c+iy+r$ of the shifted ${}_2F_1$, so that the ${}_3F_2$ collapses to a ${}_2F_1$ evaluated at $1$, summed by the Gauss formula:
\begin{equation}
\hg{2}{1}\argu{a+iy+r,\;e}{a+c+r}{1}=\frac{\Gamma(a+c+r)\,\Gamma(\alpha-\frac12-m-iy)}{\Gamma(c-iy)\,\Gamma(\alpha-\beta-1-m+r)} .
\end{equation}
Collecting all factors and recognizing the sum over $r$ as a terminating ${}_3F_2$ of unit argument yields \eqref{eq:G-closed-form}. The stated conditions ensure the convergence of all the integrals involved.
\end{proof}

\begin{proof}[Proof of Theorem \ref{thm:biorthogonality-F}]
We first treat the case $m<\alpha-\frac12$, for which $J_m=0$ and the statement reduces to \eqref{eq:biorthogonality-F-simple}.
Fix $0\leq  m<\alpha-\frac12$ and choose first $\mu$ such that $\mathrm{Re}(\mu+\alpha)<1$ this is possible since $\mu$ is a free parameter of the basis $\{\ket{w_k^-}\}$ and the assertion of the theorem does not involve it. Since the form $\bra{{d^-_n}^*|}$ is supported on $\bra{0^\dagger}|,\dots,\bra{n^\dagger}|$ and the forms $\bra{{w_k^-}^\dagger|}$, $0\le k\le n$, constitute a triangular basis of the space of such forms, there are unique coefficients $g_{n,k}=\bra {d_n^-}^\star\ket{w_k^-}$ such that
\begin{equation}
\bra{{d^-_n}^*}|=\sum_{k=0}^n g_{n,k}\bra{{w_k^-}^\dagger|},
\qquad\text{leading to}\qquad
\mathcal{F}_n(x)=\sum_{k=0}^n g_{n,k}\,\hat R_k(x).
\end{equation}
Using \eqref{eq:resolutionIdentity}, \eqref{eq:hhat} and Lemma \ref{lem:G-closed-form}, the interchange of summation and integration being justified by the convergence conditions,
\begin{equation}
\int_0^{\infty}\mathcal{F}_n\,G_m\,d\sigma_\mu
=\sum_{k=0}^{n} g_{n,k}\,\Omega_{k,m}
=\bra{{d^-_n}^*}|Z\ket{d^-_m}=\kappa_n\,\delta_{n,m},
\end{equation}
the last equality being obtained using Equation \eqref{eq:InnerProductGEVP}. It remains to substitute the closed form \eqref{eq:G-closed-form} of $G_m$. Since
\begin{equation}
\left|\Gamma(\alpha-\tfrac12-m+iy)\right|^{2}=\frac{\left|\Gamma(\alpha-\tfrac12+iy)\right|^{2}}{(x+\alpha-\beta-1-m)_m\,(\alpha+\beta-m-x)_m}
\end{equation}
on the line $x=\beta+\frac12+iy$, one has
\begin{equation}
G_m(y)\,d\sigma_\mu(y)=\frac{(-1)^{m+1}}{\Gamma(\eta-2\alpha)\Gamma(\eta-\alpha+\beta)\Gamma(\alpha-\beta-1-m)}\;\mathcal{G}_m(x)\,\frac{W(y)}{2\pi}\,dy,
\end{equation}
the factors $|\Gamma(c+iy)|^{2}$ canceling between \eqref{eq:G-closed-form} and \eqref{eq:sigma-mu}. The parameter $\mu$ has completely disappeared, and the resulting identity is precisely \eqref{eq:biorthogonality-F-simple} with $h_n$ as in \eqref{eq:hn-F}, upon using $(1+\beta-\alpha)_n=(-1)^n(\alpha-\beta-n)_n$; by \eqref{eq:factorization-GW} this is \eqref{eq:biorthogonality-F} in the case at hand.

For arbitrary $m$ the argument is unchanged up to the last step. Both sides of \eqref{eq:biorthogonality-F} are, for fixed $n$ and $m$, analytic in the parameters; the left-hand side of \eqref{eq:biorthogonality-F-simple}, written as in \eqref{eq:factorization-GW}, is the continuous dual Hahn functional with parameters $(a,b,c)$ evaluated at $c=c_m$ on the function $\mathcal{F}_n\Theta_m$ which is independent of $c$. That functional is analytic in $c$ on $\mathrm{Re}\,c>0$, and its analytic continuation past $c=0$ is obtained by deforming the contour so as to keep the poles of $\Gamma(c\pm iy)$ on their original sides. The poles that must be recaptured are those at $y=\mp i(c+j)$ with $c+j<0$; writing the integral as $\frac{1}{4\pi i}\int_{-i\infty}^{i\infty}(\cdots)\,dz$ with $z=iy$, each pair contributes the residue of
\begin{equation}
\frac{\Gamma(a+z)\Gamma(b+z)\Gamma(c+z)\Gamma(a-z)\Gamma(b-z)\Gamma(c-z)}{\Gamma(2z)\Gamma(-2z)}
\end{equation}
at $z=-c-j$, which is
\begin{equation}
\frac{(-1)^j}{j!}\,\frac{\Gamma(a-c-j)\Gamma(b-c-j)\Gamma(a+c+j)\Gamma(b+c+j)\Gamma(2c+j)}{\Gamma(-2c-2j)\Gamma(2c+2j)} .
\end{equation}
Rewriting the shifted Gamma functions in terms of Pochhammer symbols through $\Gamma(a-c-j)=\Gamma(a-c)/(a-c-j)_j$, $(a-c-j)_j=(-1)^j(1+c-a)_j$, $\Gamma(-2c-2j)=\Gamma(-2c)/(1+2c)_{2j}$ and $(1+2c)_{2j}/(2c+j)_j=(2c)_j(c+1)_j/(c)_j$ turns this residue into \eqref{eq:masses}, evaluated at $c=c_m$. Since $\mathcal{F}_n\Theta_m$ is regular at these points, the continuation of the functional is the left-hand side of \eqref{eq:biorthogonality-F}, while its value $h_n\delta_{n,m}$ is unchanged, being analytic in the parameters throughout.
\end{proof}

\begin{remark}
The threshold $m=\alpha-\frac12$ is not an artefact of the method. The functions $\Gamma(c_m\pm iy)$ have poles at $y=\mp i(c_m+j)$; these are off the contour for every $m$, but they pinch the origin and cross as $c_m$ changes sign, and the masses \eqref{eq:masses} are exactly the residues they leave behind. The mass points $x_{m,j}$ lie on the quadratic lattice attached to the eigenvalues $\mu(x)$, and each of them is a pole of $\mathcal{G}_m$, which is why the factorized form \eqref{eq:factorization-GW}, in which $\Theta_m$ is regular there, is the right one in which to state the general relation.
\end{remark}

\begin{remark}
It is natural to ask how \eqref{eq:biorthogonality-F} sits with respect to the general biorthogonality of Gupta and Masson \cite{gupta1998contiguous}, which stands at the top of a scheme of rational biorthogonalities and from which many others are obtained by degeneration. The two are of different types in the following sense. Their functions are of the $R_{II}$ type of Ismail and Masson \cite{ismail1995generalized}, with two interpolation points at each step and \emph{both} biorthogonal families rational. On the other hand, the recurrence \eqref{eq:RecurrenceRelation Rational Functions F_n} is of the $R_I$ type : one interpolation point, one polynomial family and one rational family. More precisely, solving the recurrence relation \eqref{eq:RecurrenceRelation Rational Functions F_n} for $\mathcal{F}_{n+1}$ puts it in the form $\mathcal{F}_{n+1}=(z-c_n)\mathcal{F}_n-\lambda_n(z-a_n)\mathcal{F}_{n-1}$ (up to normalization) in the variable $z=\mu(x)$, with nodes $a_n=-( \alpha-n)(\alpha-n+1)$. Accordingly, the rational functions naturally associated with the second series are
\begin{equation}
\frac{\mathcal{F}_n(x)}{\prod_{j=1}^{n}(\mu(x)+(\alpha-j)(\alpha-j+1))}
=\frac{\mathcal{F}_n(x)}{(x+\alpha-\beta-n)_n\,(\alpha+\beta+1-n-x)_n},
\end{equation}
in accordance with the general theory of \cite{zhedanov1999biorthogonal}. Observe that the partner functions $\mathcal{G}_m$ of \eqref{eq:def-G} have their poles at the reflected nodes.
The $R_I$ case is the degeneration of $R_{II}$ in which one interpolation point is sent to infinity, so the present system does sit below theirs in the hierarchy, but that degeneration is not effected by a specialization of the parameters that leaves the lattice untouched. The natural setting for \eqref{eq:biorthogonality-F} is therefore the $R_I$ theory of Ismail and Masson \cite{ismail1995generalized}, of which it provides an explicit continuous instance.
\end{remark}

\section{Outlook}
The two-diagonal representations of the meta Hahn algebra on the space of sequences have been classified and shown to fall into two families, of positive and of negative type, which are connected by transposition. Each of these representations carries at once a set of orthogonal polynomials and a set of biorthogonal rational functions. Which of the two is obtained depends only on the spectral problem that is posed: the overlaps between the eigenbases of $V$ and of $W=X+\mu Z$ give the continuous Hahn and the continuous dual Hahn polynomials, while the overlaps involving the solutions of the generalized eigenvalue problem for the pair $(X,Z)$ give rational functions. In both instances the bispectral properties, i.e., the three term recurrence relations and difference equations, follow from the algebra alone, by pairing the tridiagonal actions of the generators on one basis with their actions on the other.

This is the picture that the notion of a meta algebra was designed to produce, and the continuous representations show it to be robust: the mechanism that yields the Hahn polynomials and the Hahn rational functions together in the finite-dimensional setting operates unchanged on infinite-dimensional spaces. What does change is the analytic character of the orthogonality. In the finite-dimensional case the biorthogonality is a finite sum over a discrete grid, here the two relations obtained are integrals, over a vertical line of the complex plane, with a weight built from Gamma functions. In the second case the measure acquires, beyond a threshold in the degree, finitely many mass points, exactly as the continuous dual Hahn measure does when one of its parameters becomes negative. 

Several questions are raised by these results. The $q$-analogue is the most immediate one: the meta $q$-Hahn algebra has been introduced \cite{BBPTTVZZ2026MetaqHahn} and its finite-dimensional representations analyzed, so that the infinite dimensional representations should similarly produce the continuous dual $q$-Hahn polynomials together with rational functions of $q$-Hahn type. One may also ask how far up the scheme the construction goes, that is, whether meta algebras of Racah and of Wilson type admit two-diagonal representations on the space of sequences with the same dual pair of outcomes. Finally, the notion of a Leonard trio $(V,\tilde{V},Z)$ was introduced precisely to axiomatize the situation in which the overlaps between two eigenbases are bispectral and biorthogonal rational functions attached to a generalized eigenvalue problem \cite{CGLMVW2026trios}. Leonard trios, like Leonard pairs, are at present defined on finite-dimensional spaces, and the algebra generated by the trio of Hahn type is the very algebra studied here \cite{CLMV2026trioHahn}. The representations obtained above realize the trio situation on the space of sequences instead. Whether the notion, and the classification it opens the way to, extend to that setting is a natural question.

Quadratic algebras of Hahn type arise as symmetry algebras of superintegrable models and control the bispectral problems attached to them. Whether the meta Hahn algebra, with its distinguished split bases and its generalized eigenvalue problem, plays a comparable structural role in a physical setting is a question we find worth asking.

\section*{Acknowledgments}
The authors are grateful to N.~Cramp\'e and S.~Tsujimoto for stimulating discussions. LV is funded in part by a Discovery Grant from the Natural Sciences and Engineering Research Council (NSERC) of Canada. QL enjoys a postdoctoral fellowship provided by this grant.

\section*{Statements and declarations}

\noindent\textit{Conflict of interest.} The authors have no competing interests to declare that are relevant to the content of this article.

\medskip
\noindent\textit{Data availability.} Data sharing is not applicable to this article, as no datasets were generated or analyzed during the current study.

\bibliographystyle{unsrt} 

\bibliography{Biblio.bib}

\end{document}